\documentclass{article}
\usepackage{amssymb,amsmath,bm,geometry,graphics,graphicx,url,color}
\usepackage{amsfonts}
\usepackage{graphicx}
\usepackage{multirow}
\usepackage{amscd}
\usepackage{enumerate}
\usepackage{mathrsfs}
\usepackage{xypic}
\usepackage{stmaryrd}
\usepackage{fancybox}
\usepackage{exscale}
\usepackage{enumitem,array}%
\usepackage{graphicx}

\def\pdt2{\partial_t^2}
\def\pdx2{\partial_x^2}
\newcommand{\norm}[1]{\left\Vert#1\right\Vert}
\newcommand{\normmm}[1]{{\left\vert\kern-0.25ex\left\vert\kern-0.25ex\left\vert #1
    \right\vert\kern-0.25ex\right\vert\kern-0.25ex\right\vert}}

\newtheorem{theo}{Theorem}[section]

\newtheorem{rem}[theo]{Remark}
\newtheorem{defi}[theo]{Definition}

\def\no{\noindent}

\title{Symmetric and symplectic exponential integrators
 for nonlinear Hamiltonian systems}

\author{Yajun Wu\,\footnote{School of Mathematical Sciences, Qufu Normal
University, Qufu  273165,  P.R.China. E-mail:~{\tt
1921170786@qq.com}} \and Bin Wang\thanks{School of Mathematical
Sciences, Qufu Normal University, Qufu  273165, P.R.China;
Mathematisches Institut, University of T\"{u}bingen, Auf der
Morgenstelle 10, 72076 T\"{u}bingen, Germany. E-mail:~{\tt
wang@na.uni-tuebingen.de}} }

\begin{document}
\maketitle
\begin{abstract}
This letter studies symmetric and symplectic exponential integrators
when applied  to numerically computing nonlinear Hamiltonian
systems.   We first establish the symmetry and symplecticity
conditions of exponential integrators and then show that these
conditions are extensions of the symmetry and symplecticity
conditions  of Runge-Kutta methods. Based on these conditions, some
symmetric and symplectic exponential integrators up to order four
are derived. Two numerical experiments are carried out and the
results demonstrate the remarkable numerical behavior of the new
 exponential integrators in comparison with
some symmetric and symplectic Runge-Kutta methods in the literature.
\medskip

\no{Keywords:} exponential integrators; symmetric methods;
symplectic methods; Hamiltonian systems

\medskip
\no{MSC (2000):} 65L05, 65P10
\end{abstract}

\section{Introduction}
 In this letter, we explore  efficient symmetric and symplectic methods for solving the   initial
 value problems expressed in the following from
\begin{equation}
y'(t)=My(t)+f(y(t)),\qquad t\in [t_0,T],\qquad
y(t_0)=y_0,\label{prob}
\end{equation}
where $(-M)$ is assumed to be a  linear operator on a Banach space
$X$ with a norm $\norm{\cdot}$, $M$ is the infinitesimal generator
of a strongly continuous semigroup $e^{tM}$ on $X$  and the function
$f: \mathbb{R}^{d}\rightarrow \mathbb{R}^{d}$ is analytic (see, e.g.
\cite{Hochbruck2010}). It follows from the assumption of $M$   that
there exist  two constants $C$ and $\omega$  such that
 \begin{equation}\norm{e^{tM}}_{X\leftarrow X}\leq Ce^{\omega t},\ \ \ \ \ t\geq0.\label{condition A}%
\end{equation} We
note that the linear operator $-M$ can be   a $d\times d$ matrix if
$X$ is chosen as $X=\mathbb{R}^{d}$ or $X=\mathbb{C}^{d}$. Under
this situation, $e^{tM}$  is accordingly  the matrix exponential
function.    It is known that  the exact solution of \eqref{prob}
can be represented by the variation-of-constants formula
  \begin{equation}
y(t)= e^{tM}y_0+  \int_{0}^t e^{(t-\tau)M}f(y(\tau))d\tau.\\
\label{probsystemF3}%
\end{equation}


Problems of the form \eqref{prob} often arise in a wide range of
practical applications such as quantum physics, engineering,
flexible body dynamics, mechanics, circuit simulations and other
applied sciences (see, e.g.
\cite{Grimm2006,Hochbruck2010,Wang2018,wu-book}). Some highly
oscillatory problems,  Schr\"{o}dinger equations, parabolic partial
differential equations with their spatial discretisations all fit
the form. In order to solve \eqref{prob} effectively, many
researches have been done and the readers are referred to
\cite{Iserles2002a,Iserles2002b,Kassam2005,Khanamiryan2008,Krogstad2003}
for example. Among them, a standard form of exponential integrators
is formulated and these integrators have been studied by many
researchers. We refer the
 reader to \cite{Berland-2005,Caliari-2009,Calvo-2006,Celledoni-2008,Grimm2006,Hochbruck2009,Ostermann2006,wang-2016,IMA2018,Wang2017}
for some examples on this topic  and a systematic survey of
exponential integrators is referred to \cite{Hochbruck2010}.

On the other hand,  it can be observed that the problem \eqref{prob}
can become  a nonlinear  Hamiltonian system if
$$f(y)=J^{-1}\nabla U(y),\ \ M=J^{-1}Q,$$ where $U (y )$ is a smooth
potential function, $Q$ is a symmetric matrix, and $J=\left(
                             \begin{array}{cc}
                               0 & I \\
                               -I& 0 \\
                             \end{array}
                           \right)
$ with   the identity $I$. The energy of this Hamiltonian system is
$H(y)=\dfrac{1}{2}y^{\intercal}Qy+U(y)$. For this system, symplectic
exponential integrators (see \cite{WuMei2017JCP}) are strongly
recommended since they can preserve the symplecticity of the
original problems  and  provide good long time energy preservation
and stability.
 Besides, it is shown in \cite{hairer2006} that symmetric methods also have
excellent long time behaviour when applied to reversible
differential equations and symmetric exponential integrators have
been considered for solving Schr\"{o}dinger equations in
\cite{Celledoni-2008}. However, it seems that symmetric exponential
integrators have not been used for ODEs and moreover symmetric and
symplectic exponential integrators have never been studied so far,
which motives this letter.

The  main  contribution of this letter is to analyse and derive
symmetric and symplectic exponential integrators. The integrators
have symmerty and symplecticity simultaneously. The letter is
organized as follows. In Section \ref{sec2}, we present the scheme
of exponential integrators and derive its properties including the
symmetry and symplecticity conditions. Then we are devoted to the
construction of some practical symmetric and symplectic exponential
integrators in Section \ref{sec3}.  In Section \ref{sec4}, we carry
out two numerical experiments and the numerical results demonstrate
the remarkable efficiency of the new integrators in comparison with
some existing methods in the scientific literature. The last section
is concerned with conclusions.

\section{Exponential integrators and their properties}\label{sec2} In this section, we first present
the scheme of exponential integrators and then analyze their
symmetry and symplecticity conditions.
\begin{defi} (See \cite{Hochbruck2010}.)
An $s$-stage exponential integrator (EI) for the  problem
\eqref{prob} is defined by
\begin{equation}
\begin{cases}
\begin{array}[c]{ll}%
Y_{i}&=e^{c_{i}hM}y_{0}+h\textstyle\sum\limits_{j=1}^{s}\bar{a}_{ij}(hM)f(Y_{j}),\qquad\
i=1,\ldots,s,\\
y_{1}&=e^{hM}y_{0}+h\textstyle\sum\limits_{i=1}^{s}\bar{b}_{i}(hM)f(Y_{i}),
\end{array}
\end{cases}
 \label{erk dingyi}%
\end{equation}
where $c_i$ are constants and $\bar{a}_{ij}(hM)$ and
$\bar{b}_{i}(hM)$ are matrix-valued functions of $hM$.
\end{defi}

\begin{rem}It is worth mentioning that if
$M\rightarrow0$ an exponential integrator $\eqref{erk dingyi}$
reduces to a classical Runge-Kutta (RK) method with the coefficients
$c_i, \bar{a}_{ij}(0), \bar{b}_{i}(0)$ for $i,j=1,\ldots,s$.
\end{rem}

The next theorem gives the symmetry conditions of the EI method.
\begin{theo}\label{the symmetry conditions of the ERK}%
An s-stage exponential integrator $\eqref{erk dingyi}$ is symmetric
if and only if its coefficients satisfy the following conditions:
\begin{equation}
\begin{array}[c]{ll}%
c_{i}=1-c_{s+1-i},\quad\ &i=1,2,\ldots,s,\\
\bar{a}_{ij}(hM)=e^{c_{i}hM}\bar{b}_{s+1-j}(-hM)-\bar{a}_{s+1-i,s+1-j}(-hM),\quad\ &i,j=1,2,\ldots,s,\\
\bar{b}_{i}(hM)=e^{hM}\bar{b}_{s+1-i}(-hM),\quad\ &i=1,2,\ldots,s.
\end{array}\label{erkdc}%
\end{equation}
\end{theo}
\begin{proof}
For the exponential integrator $\eqref{erk dingyi}$, exchanging
${1}\leftrightarrow{0}$ and replacing $h$ by $-h$ yields
\begin{equation}
\begin{array}[c]{ll}
\widehat{Y}_{i}&=e^{-c_{i}hM}y_{1}-h\textstyle\sum\limits_{j=1}^{s}\bar{a}_{ij}(-hM)f(\widehat{Y}_{j}),\quad\
i=1,\ldots,s,\\
y_{0}&=e^{-hM}y_{1}-h\textstyle\sum\limits_{i=1}^{s}\bar{b}_{i}(-hM)f(\widehat{Y}_{i}).
\end{array}\label{erkd}
\end{equation}
Then we have\\
\begin{equation}
\begin{array}[c]{ll}
y_{1}&=e^{hM}y_{0}+h\textstyle\sum\limits_{i=1}^{s}e^{hM}\bar{b}_{i}(-hM)f(\widehat{Y}_{i}),\\
\widehat{Y}_{i}&=e^{(1-c_{i})hM}y_{0}+h\textstyle\sum\limits_{j=1}^{s}e^{(1-c_{i})hM}\bar{b}_{j}(-hM)f(\widehat{Y_{j}})
-h\textstyle\sum\limits_{j=1}^{s}\bar{a}_{ij}(-hM)f(\widehat{Y}_{j})
\\&=e^{(1-c_{i}hM}y_{0}+h\textstyle\sum\limits_{j=1}^{s}\Big(e^{(1-c_{i})hM}\bar{b}_{j}(-hM)-\bar{a}_{ij}(-hM)\Big)f(\widehat{Y}_{j}).
\end{array}\label{erke}
\end{equation}
For the second formula of \eqref{erke} and the first formula of
$\eqref{erk dingyi}$, it is required that the following conditions
are true
$$Y_{1}=\widehat{Y}_{s},\ \ Y_{2}=\widehat{Y}_{s-1},\ \ \cdots,\ \ Y_{s}=\widehat{Y}_{1}.$$
Based on these conditions, we obtain that the following two formulae
are equal
\begin{equation}
\begin{array}{ll}
Y_{i}&=e^{c_{i}hM}y_{0}+h\textstyle\sum\limits_{j=1}^{s}\bar{a}_{ij}(hM)f(Y_{j}),\\
\widehat{Y}_{s+1-i}&=e^{(1-c_{s+1-i})hM}y_{0}+h\textstyle\sum\limits_{j=1}^{s}\Big(e^{(1-c_{s+1-i})hM}
\bar{b}_{j}(-hM)-\bar{a}_{s+1-i,j}(-hM)\Big)f(\widehat{Y}_{j}).
\end{array}\label{erkf}
\end{equation}
This implies that
\begin{equation}
\begin{array}{ll}
c_{i}=1-c_{s+1-i}, \quad\ &i=1,2,\ldots,s,\\
\bar{a}_{ij}(hM)=e^{(1-c_{s+1-i})hM}\bar{b}_{s+1-j}(-hM)-\bar{a}_{s+1-i,s+1-j}(-hM),\quad\
&i,j=1,2,\ldots,s.
\end{array}
\end{equation}\\
Again, according to the second formula of $\eqref{erk dingyi}$ and
the first formula of $\eqref{erke}$, we obtain the third result of
\eqref{erkdc}.
 Therefore, the exponential integrator $\eqref{erk dingyi}$
is symmetric if and only if the conditions \eqref{erkdc} hold.
\end{proof}

\begin{rem}
It is noted that when $M=0$, these symmetry conditions become
\begin{equation}
\begin{array}[c]{ll}%
c_{i}=1-c_{s+1-i},\quad\ &i=1,2,\ldots,s,\\
\bar{a}_{ij}(0)=\bar{b}_{s+1-j}(0)-\bar{a}_{s+1-i,s+1-j}(0),\quad\ &i,j=1,2,\ldots,s,\\
\bar{b}_{i}(0)=\bar{b}_{s+1-i}(0),\quad\ &i=1,2,\ldots,s,
\end{array}\label{rkdc}%
\end{equation}
which are the exact symmetry conditions  of $s$-stage RK methods.
\end{rem}

About the symplecticity conditions of exponential integrators, we
have the following result.
\begin{theo}\label{symplecticity conditions of the ERK}%
(See \cite{WuMei2017JCP}.) If the coefficients of an $s$-stage
exponential integrator $\eqref{erk dingyi}$  satisfy
\begin{equation}
\begin{array}[c]{ll}%
\bar{b}_{i}(hM)^{T}JSS_i^{-1}=S_i^{-T}S^{T}J\bar{b}_{i}(hM)=\gamma J,\quad\ &\gamma \in \mathbb{R},\ \ i=1,2,\ldots,s,\\
\bar{b}_{i}(hM)^{T}J\bar{b}_{j}(hM)=\bar{b}_{i}(hM)^{T}JSS_i^{-1}\bar{a}_{ij}(hM)+\bar{a}_{ji}(hM)^{T}S_j^{-T}S^{T}J\bar{b}_{j}(hM),
\quad\ &i,j=1,2,\ldots,s,
\end{array}\label{erkdcss}%
\end{equation}
where  $S = e^{hM}$ and $S_i =e^{c_{i}hM}$ for $i = 1,\ldots,s,$
then the integrator is symplectic.
\end{theo}

\begin{rem}
We also remark that when $M=0$, these conditions reduce to
\begin{equation}
\begin{array}[c]{ll}%
&\bar{b}_{i}(0)\bar{b}_{j}(0)=\bar{b}_{i}(0)\bar{a}_{ij}(0)+\bar{b}_{j}(0)\bar{a}_{ji}(0),
\quad\ i,j=1,2,\ldots,s,
\end{array}\label{rkdcss}%
\end{equation}
which are the exact symplecticity conditions  of $s$-stage RK
methods.
\end{rem}

\section{Symmetric and symplectic EI}\label{sec3}
In this section, we derive a class of symmetric and symplectic
exponential integrators. To this end, we consider the following
special exponential integrators.

\begin{defi}
\label{scheme EI spe} (See \cite {WuMei2017JCP}) Define a special
kind of
$s$-stage exponential integrators   by%
\begin{equation}
\begin{aligned} \bar{a}_{ij}(h M)=a_{ij}e^{(c_i-c_j) h M},\ \ \bar{b}_{i}(h M)=b_{i}e^{(1-c_i)h
M},\ \ i,j=1,\ldots,s,
\end{aligned}
\label{rev co}%
\end{equation}
where
\begin{equation}c=(c_1,\ldots,c_s)^{\intercal},\ \
b=(b_1,\ldots,b_s)^{\intercal},\ \
A=(a_{ij})_{s\times s}\label{RK co}%
\end{equation} are   the coefficients of an $s$-stage  RK
method. We denote this class of exponential integrators  by SEI.
\end{defi}

For these special exponential integrators,   the following
properties can be derived.
\begin{theo}\label{ERK condition}
\begin{itemize}
\item If the $s$-stage RK method \eqref{RK co} is symmetric, then  the   $s$-stage
SEI \eqref{rev co} is also symmetric.

\item The  SEI \eqref{rev co}  is symplectic if the  RK method \eqref{RK co} is symplectic.

\item The  SEI \eqref{rev co}  is symmetric and symplectic if the  RK method \eqref{RK co} is symmetric and symplectic.

\item If  the $s$-stage RK method \eqref{RK co} is of order $p$, then the  SEI \eqref{rev co}  is also of order
$p$.
 \end{itemize}
\end{theo}
\begin{proof}
Inserting \eqref{rev co} into the symmetry  conditions of
\eqref{erkdc} yields
\begin{equation*}
\begin{array}[c]{ll}%
c_{i}&=1-c_{s+1-i},\\
 a_{ij}e^{(c_i-c_j) h
M}&=e^{c_{i}hM}b_{s+1-j}e^{-(1-c_{s+1-j})h
M}-a_{s+1-i,s+1-j}e^{-(c_{s+1-i}-c_{s+1-j}) h
M}\\&=(b_{s+1-j}-a_{s+1-i,s+1-j})e^{(c_i-c_j) h
M},\\
b_{i}e^{(1-c_i)h M}&=e^{hM}b_{s+1-i}e^{-(1-c_{s+1-i})h M}=
b_{s+1-i}e^{(1-c_i)h M},
\end{array}
\end{equation*}
which can be simplified as the symmetry conditions \eqref{rkdc} of
RK methods. Thus the first statement is true. The second statement
can be obtained immediately by considering Theorem 3.2 of
\cite{WuMei2017JCP}. Based on the above two results, the third one
holds. The last result comes from Theorem 3.1 of
\cite{WuMei2017JCP}.
\end{proof}

In what follows, based on Theorem \ref{ERK condition} we construct
some practical symmetric and symplectic  SEI integrators.


\subsection{One-stage symmetric and
symplectic SEI} First consider an one-stage RK method with the
coefficients: \[%
\begin{tabular}
[c]{l}%
\begin{tabular}
[c]{c|c}%
$c_{1}$ & $a_{11}$\\\hline
& $\raisebox{-1.3ex}[0.5pt]{$ b_{1}$}$%
\end{tabular}
\end{tabular}%
\]
According to \eqref{rkdc} and \eqref{rkdcss}, this method is
symmetric and symplectic if
\begin{equation}
\begin{array}[c]{ll}%
c_1=1/2,\  a_{11}=b_1-a_{11},\ b_1^2=2b_1a_{11}.
\end{array}\label{one stage}%
\end{equation}
From these formulae, it follows that
\begin{equation}
\begin{array}[c]{ll}%
c_1=1/2,\ \ a_{11}=\frac{1}{2}b_1.
\end{array}\label{erkneighteen}
\end{equation}
This gives a class of symmetric and symplectic exponential
integrators by considering \eqref{rev co}. As an example, we choose
$b_1=1$ and denote the method as SSSEI1s2. It can be checked that
this   RK method is implicit midpoint rule. Thus the  symmetric and
symplectic  SEI is of order two.
\subsection{Two-stage symmetric and
symplectic SEI} Consider a two-stage  RK method whose coefficients
are given
 by a Butcher tableau:
 \[%
\begin{tabular}
[c]{c|cc}%
$c_{1}$ & $a_{11}$& $a_{12}$ \\
$c_{2}$ & $a_{21}$ & $a_{22}$\\\hline &
$\raisebox{-1.3ex}[0.5pt]{$b_1$}$ &
$\raisebox{-1.3ex}[0.5pt]{$b_2$}$\\ &
\end{tabular}
\]
The symmetry conditions of this method are
\begin{equation}
\begin{array}{ll}%
c_{1}=1-c_{2},\ b_1=b_2,\ \ a_{11}=b_2-a_{22},\ a_{12}=b_1-a_{21}, \
a_{21}=b_2-a_{12},\ a_{22}=b_1-a_{11}.
\end{array}\label{erkschemell}
\end{equation}
The RK method is symplectic if
\begin{equation}
\begin{array}[c]{ll}%
&b_1^2=2b_1a_{11},\ \ b_1b_2=b_1a_{12}+b_2a_{21},\ \
b_2^2=2b_2a_{22}.
\end{array}\label{erkschemelo}
\end{equation}
According to \eqref{erkschemell} and \eqref{erkschemelo}, we obtain
\begin{equation}
\begin{array}[c]{ll}%
c_1=1-c_2,\ b_1=b_2,\ \ a_{11}=a_{22}=b_1/2,\  a_{12}+a_{21}=b_1.
\end{array}\label{erkn22}
\end{equation}
In the light of the third-order conditions of RK methods (see
\cite{hairer2006})
\begin{equation}
\begin{array}[c]{ll}%
a_{11}+a_{12}=c_1,\ &a_{21}+a_{22}=c_2,\ b_1+b_2=1,\
b_1c_1+b_2c_2=1/2,\\ b_1c^2_1+b_2c^2_2=1/3,\ &b_1(a_{11}c_1
+a_{12}c_2)+b_2(a_{21}c_1+a_{22}c_2)=1/6,
\end{array}\label{erkn23}
\end{equation}
 we choose
the parameters by
\begin{equation}
\begin{array}[c]{ll}%
c_1=\frac{3-\sqrt{3}}{6},\ b_1=1/2,\ a_{12}=\frac{3-2\sqrt{3}}{12}.
\end{array}\label{erkn22-1}
\end{equation}
This choice as well as \eqref{erkn22} and \eqref{rev co} gives an
symmetric and symplectic SEI. Moreover, it can be seen that the
corresponding RK method is Gauss method of order four (see
\cite{hairer2006}). Therefore, the SEI is also of order four, which
is denoted by SSSEI2s4.

\subsection{Three-stage symmetric and
symplectic SEI} We turn to considering   three-stage symmetric and
symplectic integrators. The following Butcher tableau
 describes a three-stage  RK method:
\[%
\begin{tabular}
[c]{c|ccc}%
$c_{1}$ & $a_{11}$ & $a_{12}$ & $a_{13}$\\
$c_{2}$ & $a_{21}$ & $a_{22}$& $a_{23}$\\
$c_{3}$ & $a_{31}$ & $a_{32}$ & $a_{33}$\\\hline &
$\raisebox{-1.3ex}[0.5pt]{$b_1$}$ &
$\raisebox{-1.3ex}[0.5pt]{$b_2$}$ &
$\raisebox{-1.3ex}[0.5pt]{$b_3$}$\\ &
\end{tabular}
\]
 This method is symmetric if the following conditions are true
\begin{equation}
\begin{array}[c]{ll}%
c_1=1-c_3,\
& c_2=1/2,\\
b_1=b_3,\
&b_1=a_{31}+a_{13}=a_{21}+a_{23}=a_{11}+a_{33},\\
b_2=a_{32}+a_{12}=2a_{22},\ &
b_3=a_{33}+a_{11}=a_{23}+a_{21}=a_{13}+a_{31}.
\end{array}\label{erkschemmellla}
\end{equation}
The RK methods are symplectic if
\begin{equation}\label{sym3}
\begin{array}[c]{lll}
&b_{1}a_{11}+b_{1}a_{11}=b_{1}^{2},\ \ \ \ b_{2}a_{21}+b_{1}a_{12}=b_{1}b_{2}, \ \ \ b_{3}a_{31}+b_{1}a_{13}=b_{1}b_{3},\\
&b_{2}a_{22}+b_{2}a_{22}=b_{2}^{2}, \ \ \ \
b_{3}a_{33}+b_{3}a_{33}=b_{3}^{2}, \ \ \ \ \ \
b_{3}a_{32}+b_{2}a_{23}=b_{2}b_{3}.
\end{array}
\end{equation}
By the formulae \eqref{erkschemmellla} and   \eqref{sym3}, the
coefficients can be given as
\[%
\begin{tabular}
[c]{c|ccc}%
$c_{1}$ & $\frac{b_1}{2}$ & $0$ & $0$\\
$\frac{1}{2}$ & $b_1$ & $\frac{b_2}{2}$& $0$\\
$1-c_{1}$ & $b_1$ & $b_2$ & $\frac{b_1}{2}$\\\hline &
$\raisebox{-1.3ex}[0.5pt]{$b_1$}$ &
$\raisebox{-1.3ex}[0.5pt]{$b_2$}$ &
$\raisebox{-1.3ex}[0.5pt]{$b_1$}$\\ &
\end{tabular}
\]
This result as well as \eqref{rev co} yields a class of symmetric
and symplectic SEI. As an example and following \cite{Sanz-Serna91},
we consider
$$c_1=\frac{8-2\sqrt[3]{2}-\sqrt[3]{4}}{12},\ \ b_1=\frac{4+2\sqrt[3]{2}+\sqrt[3]{4}}{6},\ \ b_{2}=\frac{-1-2\sqrt[3]{2}-\sqrt[3]{4}}{3}$$ and denote the method as SSSEI3s4.
 According to the fourth-order conditions of RK methods (see \cite{hairer2006})
\begin{equation}
\begin{array}[c]{ll}%
a_{11}+a_{12}+a_{13}=c_1,\ \ a_{21}+a_{22}+a_{23}=c_2,\ \
a_{31}+a_{32}+a_{33}=c_3,\ \ b_1+b_2+b_3=1,\\
b_1c_1+b_2c_2+b_3c_3=1/2,\ \ b_1c^2_1+b_2c^2_2+b_3c^2_3=1/3,\ \
b_1c^3_1+b_2c^3_2+b_3c^3_3=1/4,
\\b_1(a_{11}c_1+a_{12}c_2+a_{13}c_3)+b_2(a_{21}c_1+a_{22}c_2+a_{23}c_3) +b_3(a_{31}c_1+a_{32}c_2+a_{33}c_3)=1/6,\\
b_1c_1(a_{11}c_1+a_{12}c_2+a_{13}c_3)+b_2c_2(a_{21}c_1+a_{22}c_2+a_{23}c_3)+b_3c_3(a_{31}c_1+a_{32}c_2+a_{33}c_3)=1/8,\\
b_1(a_{11}c^2_1+a_{12}c^2_2+a_{13}c^2_3)+b_2(a_{21}c^2_1+a_{22}c^2_2+a_{23}c^2_3)+b_3(a_{31}c^2_1+a_{32}c^2_2+a_{33}c^2_3)=1/12,\\
b_1a_{11}(a_{11}c_1+a_{12}c_2+a_{13}c_3)+b_1a_{12}(a_{21}c_1+a_{22}c_2+a_{23}c_3)
+b_1a_{13}(a_{31}c_1+a_{32}c_2+a_{33}c_3)\\
\quad
+b_2a_{21}(a_{11}c_1+a_{12}c_2+a_{13}c_3)  +b_2a_{22}(a_{21}c_1+a_{22}c_2+a_{23}c_3)+b_2a_{23}(a_{31}c_1+a_{32}c_2+a_{33}c_3)\\
\quad
+b_3a_{31}(a_{11}c_1+a_{12}c_2+a_{13}c_3)+b_3a_{32}(a_{21}c_1+a_{22}c_2+a_{23}c_3)
+b_3a_{33}(a_{31}c_1+a_{32}c_2+a_{33}c_3)=1/24,
\end{array}\label{erkschemelllb}
\end{equation}
it can be  checked that the coefficients of the RK method satisfy
all the conditions. Thus this RK method is of order four and the
symmetric and symplectic SEI has same order. This symmetric and
symplectic SEI is denoted by SSSEI3s4.

\section{Numerical experiments} \label{sec4}
\label{sec:MSERKN5}
 This section presents two numerical
 experiments to
 show the remarkable efficiency  of the new integrators  as compared with some existing
 RK methods. The integrators  for comparisons are chosen as:
\begin{itemize}
  \item SSSEI1s2: the one-stage symmetric and symplectic EI of order two presented in this letter;
  \item SSSEI2s4: the two-stage symmetric and symplectic EI of order four presented in this letter;
    \item SSSEI3s4: the three-stage symmetric and symplectic EI of order four presented in this letter;
  \item SSRK1s2: the one-stage symmetric and symplectic RK method of order two obtained by letting $M=0$ for SSSEI1s2 (implicit
midpoint rule);
  \item SSRK2s4: the two-stage symmetric and symplectic RK method of order four obtained by letting $M=0$ for SSSEI2s4
  (Gauss method of order four);
    \item SSRK3s4: the three-stage symmetric and symplectic RK method of order four obtained by letting $M=0$ for
    SSSEI3s4 (the method was given in \cite{Sanz-Serna91}).
\end{itemize}

\textbf{Problem 1.} As the first numerical example, we consider the
Duffing equation defined by
\[
\left(
\begin{array}
[c]{c}%
q\\
p
\end{array}
\right)^{\prime}%
= \left(
\begin{array}
[c]{c}%
0\ \ \ \ \ \ \ \ \ \ \ \ \ \ \ 1\\
-\omega^2-k^2\ \ \ 0
\end{array}
\right) \left(
\begin{array}
[c]{c}%
q\\
p
\end{array}
\right) + \left(
\begin{array}
[c]{c}%
0\\
2k^2q^3
\end{array}
\right), \ \ \  \left(
\begin{array}
[c]{c}%
q(0)\\
p(0)
\end{array}
\right) = \left(
\begin{array}
[c]{c}%
0\\
\omega
\end{array}
\right).
\]
  It is a Hamiltonian system with the Hamiltonian $
H(q,q)=\dfrac{1}{2}p^{2}+\dfrac{1}{2}(\omega^2+k^2)q^2-\dfrac{k^2}{2}q^4.$
The exact solution of this system is $q(t)=sn(\omega t;k/\omega)$
with the Jacobi elliptic function $sn$. For this problem, we choose
$k=0.07$, $t_{\mathrm{end}}=20$, $\omega=20$ and $h=\dfrac{1}{2^i}$
for $i=3,4,5,6.$ The efficiency curves are shown in Figure
\ref{fig:problem11} (i). We integrate this problem with a fixed
stepsize  $h=1/10$ in the interval $[0,10^{i}]$ for  $i=0,1,2,3$.
The results of energy conservation are presented in Figure
\ref{fig:problem11} (ii).
\begin{figure}[ptbh]
\centering\tabcolsep=1mm
\begin{tabular}
[c]{ccc}%
\includegraphics[width=6cm,height=6cm]
{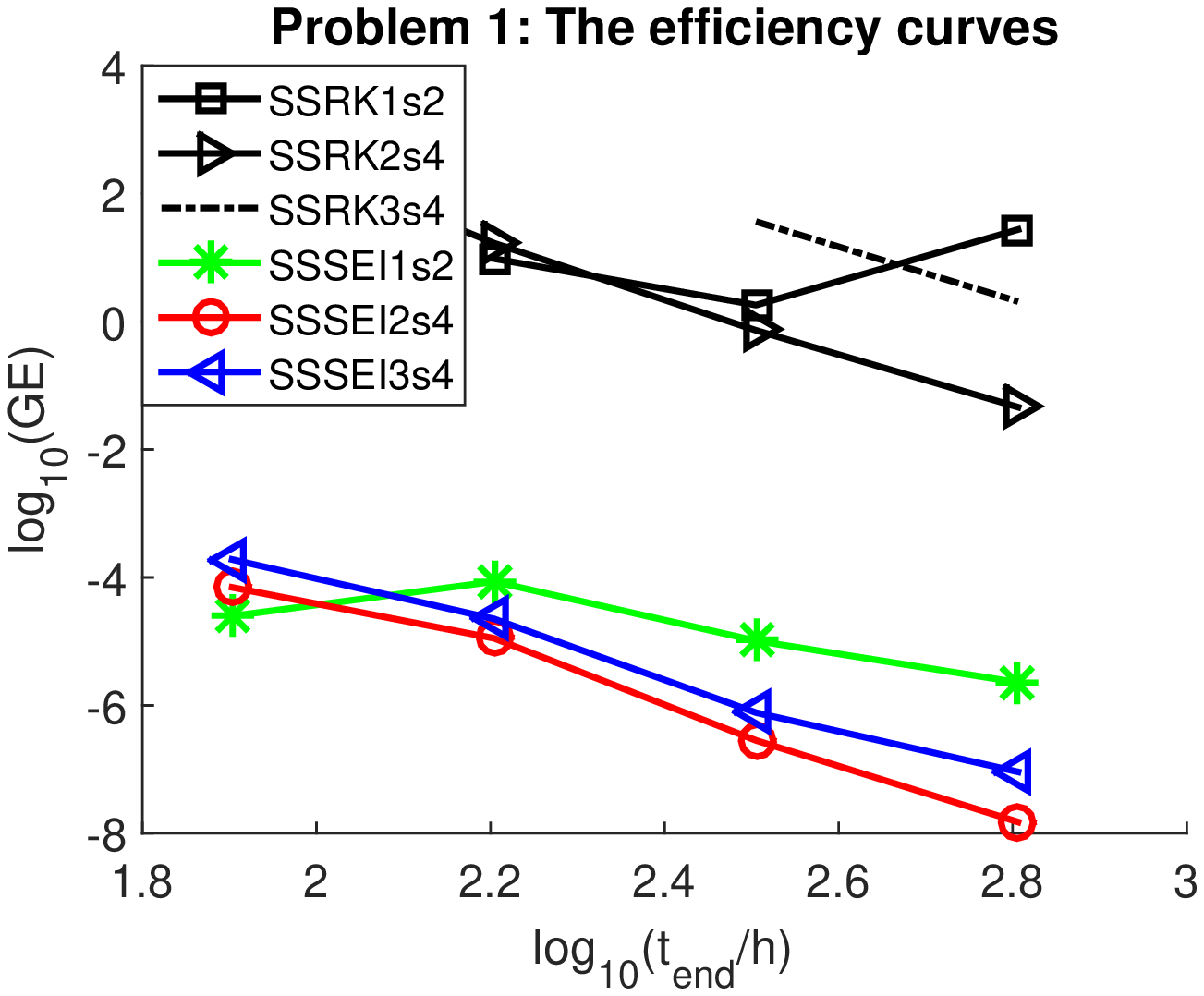} & \includegraphics[width=6cm,height=6cm] {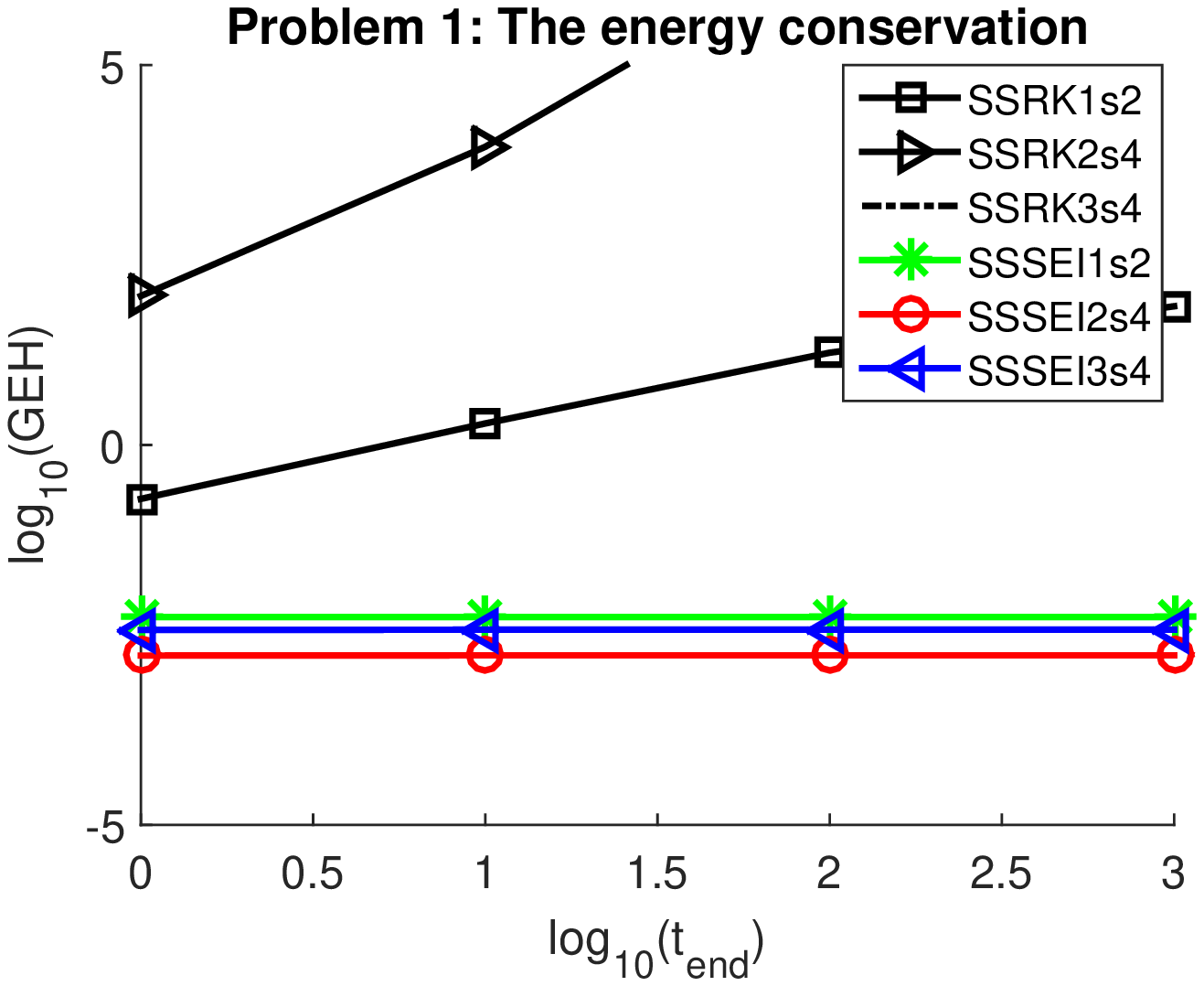} &
\\
{\small (i)} & {\small (ii)} %
\end{tabular}
\caption{(i): The logarithm of the global error ($GE$) over the
integration interval against $t_{end}/h$. (ii):\ The logarithm of
the maximum global error of Hamiltonian
$GEH=\max|H_{n}-H_{0}|$ against $\log_{10}(t_{\mathrm{end}})$.}%
\label{fig:problem11}%
\end{figure}

\textbf{Problem 2.} The second numerical example is  the following
averaged system in wind-induced
 oscillation
\begin{equation*}
\begin{aligned}& \left(
                   \begin{array}{c}
                     x_1 \\
                      x_2 \\
                   \end{array}
                 \right)
'= \left(
    \begin{array}{cc}
      -\zeta& -\lambda\\
       \lambda & -\zeta \\
    \end{array}
  \right)\left(
                   \begin{array}{c}
                     x_1 \\
                      x_2 \\
                   \end{array}
                 \right)+
\left(
                                                                           \begin{array}{c}
                                                                           x_1x_2\\
\frac{1}{2}(x_1^2-x_2^2)
                                                                           \end{array}
                                                                         \right),
\end{aligned}\end{equation*}
 where  $\lambda=r\sin(\theta)$ is a detuning parameter and $\zeta= r\cos(\theta) \geq
0$ is a damping factor with $r \geq 0,\ 0 \leq\theta \leq \pi /2.$
The first integral (when $\theta =\pi/2$) or Lyapunov function (when
$\theta <\pi/2$)  of this system  is
$$H=\frac{1}{2}r(x_1^2+x_2^2)-\frac{1}{2}\sin(\theta)\big(x_1x_2^2-\frac{1}{3}x_1^3\big)+\frac{1}{2}\cos(\theta)\big(-x_1^2x_2+\frac{1}{3}x_2^3\big).$$
 We choose the  initial values   $x_1(0)=0,\ x_2(0)=1.$ Firstly we consider  $\theta=\pi/2,\ r=20$  and  solve the problem on the
interval $[0,10]$ with  $h=\frac{1}{2^i}$ for $i=3,4,5,6$. The
efficiency curves are shown in Figure \ref{fig:problem21} (i). Then
this problem is integrated with  $h=\dfrac{1}{20}$ on the interval
$[0,10^{i}],\
 i=0,1,2,3.$ See Figure \ref{fig:problem21} (ii) for the energy
conservation for different methods.  Secondly we  choose
$\theta=\pi/2-10^{-4}$  and the efficiency curves are shown in
Figure \ref{fig:problem21} (iii) on $[0,10]$ with
$h=\frac{1}{2^{i}}, \ i=3,4,5,6.$
\begin{figure}[ptbh]
\centering\tabcolsep=1mm
\begin{tabular}
[c]{ccc}%
\includegraphics[width=5cm,height=6cm]
{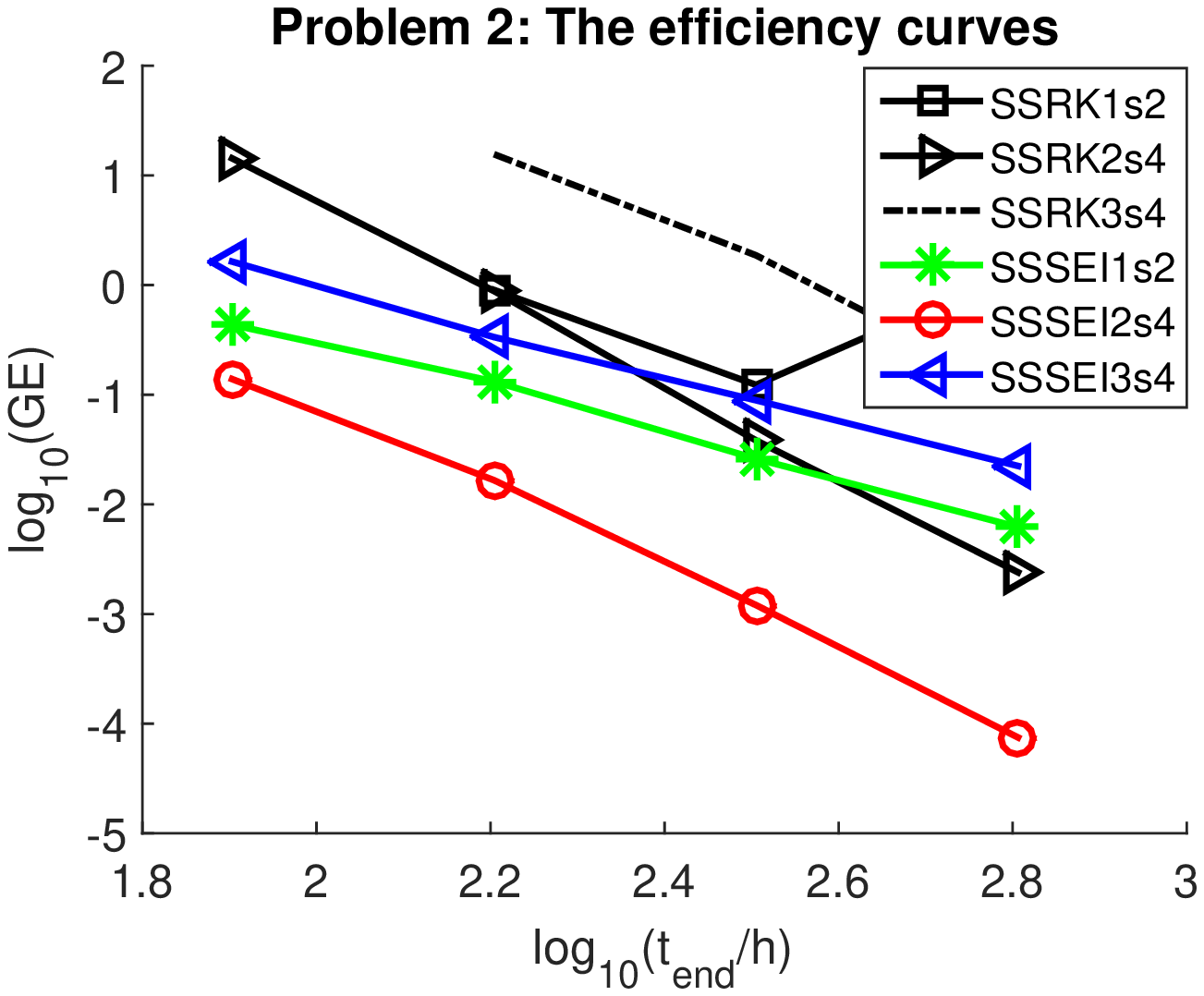} & \includegraphics[width=5cm,height=6cm] {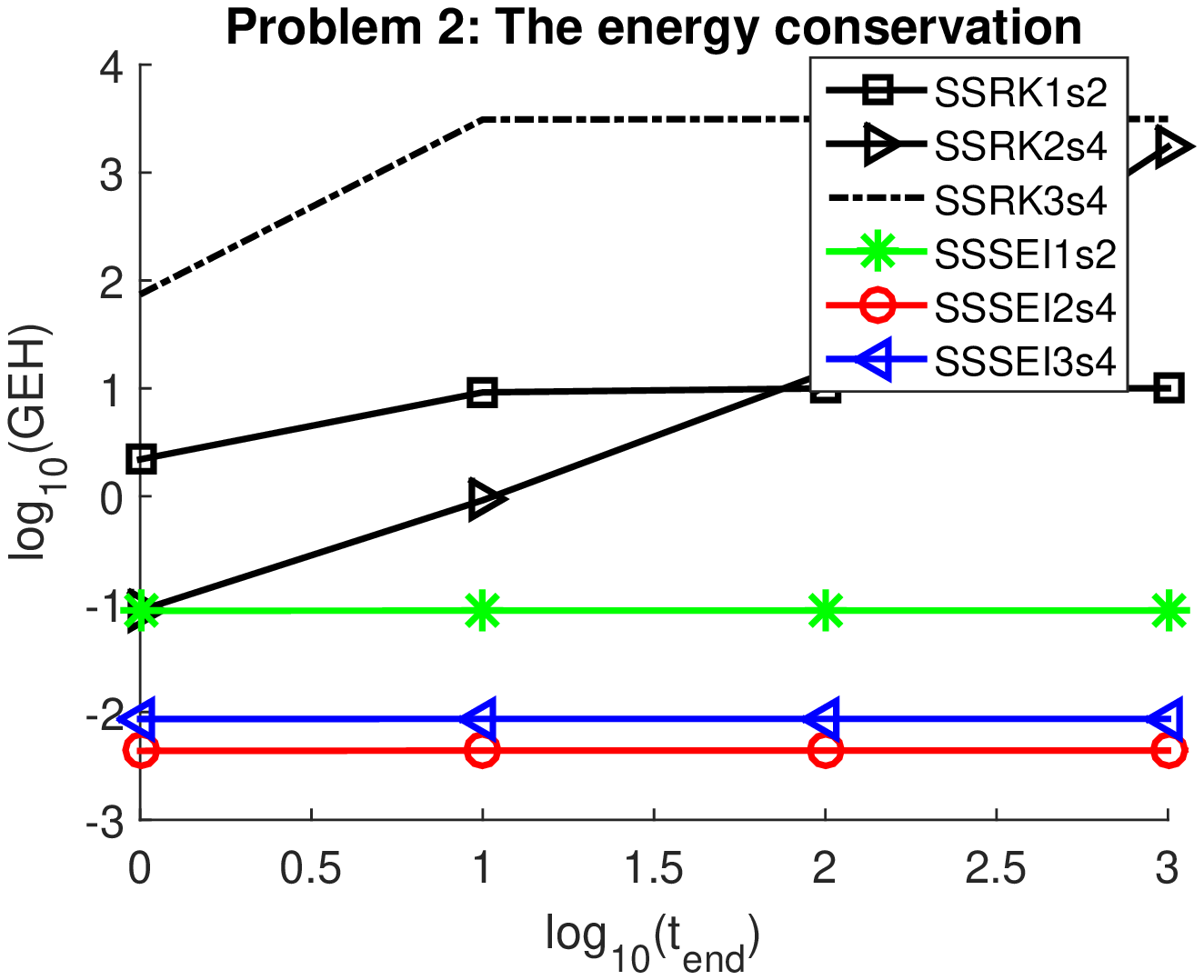} &
\includegraphics[width=5cm,height=6cm]
{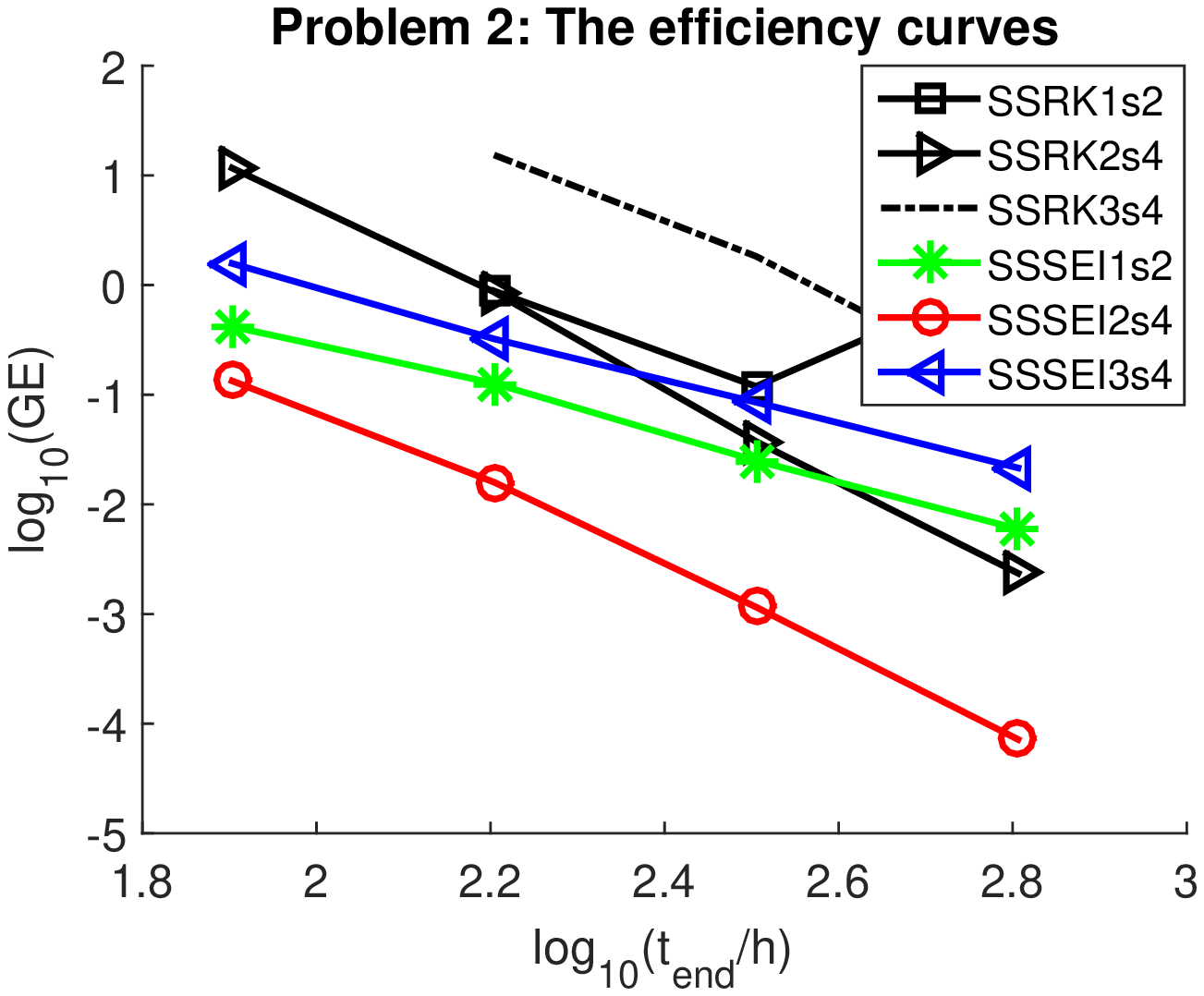}\\
{\small (i)} & {\small (ii)} & {\small (iii)}%
\end{tabular}
\caption{(i): The logarithm of the global error ($GE$) over the
integration interval against $t_{end}/h$. (ii):\ The logarithm of
the maximum global error of Hamiltonian $GEH=\max|H_{n}-H_{0}|$
against $\log_{10}(t_{\mathrm{end}})$. (iii): The logarithm of the
global error ($GE$) over the
integration interval against $t_{end}/h$.}%
\label{fig:problem21}%
\end{figure}

From the numerical results, it follows clearly that the symmetric
and symplectic exponential integrators behave   much better than
 symmetric and symplectic RK methods.

\section{Conclusions and discussions}
\label{sec5}

In this letter, in order to solve the differential equations
\eqref{prob} by using symmetric and symplectic methods, we present
the symmetry and symplecticity conditions for exponential
integrators. Then based on these conditions, we consider a special
kind of  exponential integrators and construct some practical
symmetric and symplectic exponential integrators. The remarkable
efficiency of the new integrators is shown by the numerical results
from two numerical experiments in comparison with some existing RK
methods in the literature.

\end{document}